\documentclass[reqno]{amsart}
\usepackage{amssymb,amsfonts,amstext,amsmath,amsthm}
\usepackage{amssymb}
\usepackage{enumerate}
\usepackage[colorlinks=true]{hyperref}
\hypersetup{urlcolor=blue, citecolor=red}
\newtheorem{thm}{Theorem}[section]
\newtheorem{cor}[thm]{Corollary}
\newtheorem{lem}[thm]{Lemma}
\newtheorem{prop}[thm]{Proposition}
\newtheorem{rem}[thm]{Remark}

\numberwithin{equation}{section}
\newcommand{\Om}{{\Omega}}

\newcommand{\R}{{\rm I}\!{\rm R}}

\title[Multiple solutions for an indefinite elliptic problem]{Multiple solutions for an indefinite elliptic problem with critical growth in the gradient}

\begin{document}

\vspace{1cm}

\author{Louis Jeanjean}
\address{L. Jeanjean \newline Laboratoire de Math\'ematiques
(UMR 6623), Universit\'e de Franche-Comt\'e,
16 route de Gray
25030 Besan\c{c}on Cedex, France}
\email{louis.jeanjean@univ-fcomte.fr}

\author{Humberto Ramos Quoirin}
\address{H. Ramos Quoirin \newline Universidad de Santiago de Chile, Casilla 307, Correo 2, Santiago, Chile}
\email{\tt humberto.ramos@usach.cl}

\thanks{The second author was supported by the FONDECYT project 11121567}

\begin{abstract}
We consider the problem 
\begin{equation*}
-\Delta u =c(x)u+\mu|\nabla u|^2 +f(x), \quad u \in H^1_0(\Omega) \cap L^{\infty}(\Omega),
\leqno{(P)}
\end{equation*}
where $\Omega$ is a bounded domain of $\mathbb{R}^N$, $N \geq 3$, $\mu>0$ and  $c,   f \in L^q(\Om) \text{ for some } q>\frac{N}{2}$ with $  f\gneqq 0. $
Here $c$ is allowed to change sign. We show that when $c^+ \not \equiv 0$ and $c^+ +\mu f$ is suitably small, this problem has at least two positive solutions. This result contrasts with the case $c \leq 0$, where uniqueness holds. To show this multiplicity result we first transform $(P)$ into a semilinear problem having a variational structure. Then we are led to the search of two critical points for a functional whose superquadratic part is indefinite in sign and has a so called {\it slow growth} at infinity. The key point is to show that the Palais-Smale condition holds.
\end{abstract}



\maketitle


\section{Introduction}

Let $\Omega$ be a bounded domain of $\mathbb{R}^N$ with $N \geq 3$. In this paper we are concerned with the boundary value problem
\begin{equation*}
-\Delta u =c(x)u+\mu|\nabla u|^2 +f(x), \quad u \in H^1_0(\Omega) \cap L^{\infty}(\Omega),
\leqno{(P)}
\end{equation*}
where 
$$
\mu>0,  \quad \text{and} \quad c, f \in L^q(\Om) \text{ for some } q>\frac{N}{2},\  f\gneqq 0. \leqno{(\mathcal{H})}
$$
Quasilinear elliptic equations with a gradient dependence up to the critical growth $|\nabla u|^2$ were first studied by Boccardo, Murat and Puel in the 80's \cite{BoMuPu1, BoMuPu2.5, BoMuPu3} and have been an active field of research until now, see for example \cite{AbPePr,GrMuPo,HaRu}. To situate our problem we underline that we are interested in bounded solutions. The main goal of this paper is to carry on the study of non-uniqueness of solutions for such problems, which $(P)$ is a prototype of.  \medskip

The sign of $c$ plays in $(P)$ a central role regarding  uniqueness, as well as existence, of bounded solutions. We refer to \cite{JS} for a heuristic discussion on the influence of the sign of $c$ on the nature of the problem. The case $c \leq  - \alpha_0$ a.e. in $\Omega$ for some $\alpha_0 >0$ is referred to as the coercive case. In this case, the existence of solutions holds under very general assumptions and it was shown in  \cite{BaBlGeKo,BaMu} (see also \cite{ArSe,BaPo}) that there is a unique bounded solution. When one just requires $c \leq 0$ (in particular when $c \equiv 0$) the situation is already more complex. The fact that restrictions on the data are necessary for $(P)$ to have a solution was first observed in 
\cite{FeMu1, FeMu2}. Concerning uniqueness, some partial results are given in \cite{BaBlGeKo,BaMu}, but it was only in \cite{ArDeJeTa1}  that uniqueness of bounded solutions was established under the mere condition $c \leq 0$. See
 also \cite{ArDeJeTa2} for an extension to a larger class of problems. \medskip

The case $c\gneqq 0$ started to be studied only recently. Surely in part because it was not accessible by the methods traditionally used in the coercive case. It was  shown in \cite{JS},  when $c\gneqq 0$ and for any $c, \mu$ and $f$ sufficiently small in an appropriate sense, that $(P)$ has two solutions. See also \cite{AbBi2, AbDaPe, Si}  for related results. Note that the case where $\mu$ is allowed to be non constant was treated in \cite{ArDeJeTa1} leading also, when $c\gneqq 0$  and under appropriate conditions,  to the existence of two bounded solutions. \medskip

In view of these results it remained to analyse the case where $c$ is allowed to change sign, which is the aim of the present paper. Roughly speaking we shall show that the uniqueness is lost as soon as $c^+ \not \equiv 0$, where $c^+ = \max\{0, c\}$, see Theorem \ref{t1}. \medskip

To obtain this result we first make a change a variable. It is  well-known, since \cite{KaKr}, that the change of variable $$v=\frac{1}{\mu} (e^{\mu u}-1)$$ rids the gradient term of $(P)$, reducing it to a semilinear problem with a variational structure. We consider here a slight variation of this change of variable, namely, 
\begin{equation}
\label{cv}
v=\frac{1}{\lambda} (e^{\mu u}-1),
\end{equation}
where $\lambda>0$ will be fixed later at our convenience. This  leads to the problem 
\begin{equation*}
-\Delta v -(c(x)+\mu f(x))v=c(x)g_\lambda(v) +\frac{\mu}{\lambda}f(x), \quad v \in H^1_0(\Omega) \cap L^{\infty}(\Omega),
 \leqno{(Q)}
\end{equation*}
where 
\begin{equation}
\label{gl}
g_\lambda(s)=\begin{cases} \frac{1}{\lambda} (1+\lambda s)\ln (1+\lambda s)-s & \text{ if } s \geq 0\\
0 & \text{ if } s \leq 0.\end{cases}
\end{equation}
We shall prove in Lemma \ref{pq} that if $v$ is a non negative solution of $(Q)$ then $u$ defined by \eqref{cv} is a non negative (and therefore positive, by Harnack inequality) solution of $(P)$. Solutions of $(Q)$ will  be obtained as critical points of the functional
$$I(v)=\frac{1}{2} \int_\Omega \left[|\nabla v|^2-[c(x)+\mu f(x)](v^+)^2\right]
-\int_\Omega c(x)G_\lambda(v^+) - \frac{\mu}{\lambda}\int_\Omega f(x)v$$
defined on $H_0^1(\Omega)$ and where $G_\lambda(s)=\int_0^s g_\lambda(t)\, dt$.  Note that since $f\geq 0$, critical points of $I$ are necessarily non-negative, see Lemma \ref{pq}. Since $g_{\lambda}$ behaves essentially as $s \ln(s+1)$ for $s$ large, the superquadratic part of $I$ has at infinity a growth which is usually referred to as a {\it slow superlinear growth}. \medskip

To obtain two critical points we start following the strategy used in  \cite{JS}. Note that if the positive part of $c+\mu f$ is not `too large' in a suitable sense (cf. Lemma \ref{l1}) then  $$\int_\Omega \left[|\nabla v|^2-[c(x)+\mu f(x)](v^+)^2\right]$$ is coercive. Moreover, as $G_\lambda$ is superquadratic, we shall prove that $I$ takes positive values on a sphere $\|v\|=R_\lambda$ if the lower order term coefficient $\frac{\mu}{\lambda}$ is sufficiently small. Here comes the advantage of the change of variable \eqref{cv}: we may take $\lambda$ sufficiently large so as to have $\frac{\mu}{\lambda}$ sufficiently small. This argument is possible because $G_\lambda$ grows sufficiently slowly with respect to $\lambda$, see Lemma \ref{lg}. Moreover it is easily seen that since $f \not \equiv 0$, $I$ takes negative values in the ball $B(0, R_{\lambda})$. Finally, since $c^+ \not \equiv 0$, it is possible to show that $I$ takes a negative value at some point $v_0$ outside of the ball $B(0, R_{\lambda})$. Thus $I$ has a mountain-pass geometry and it is reasonable to search for a first critical point as a minimizer of $I$ in $B(0, R_{\lambda})$ and a second one at the mountain pass level. The existence of a minimizer will follow from a standard lower semi continuity argument, whereas in proving the existence of a mountain-pass critical point we will face the difficulty of showing that Palais-Smale sequences are bounded. \medskip

We recall that the Palais-Smale condition holds for $I$ if any sequence $(u_n) \subset H^1_0(\Omega)$ such that $(I(u_n)) \subset \R$ is bounded and $||I'(u_n)||_* \to 0$ admits a convergent subsequence. The boundedness of such sequences proves to be a delicate issue due to the fact that $c$ is sign-changing and $g_{\lambda}$ has a slow growth at infinity and thus in particular does not satisfy an Ambrosetti-Rabinowitz type condition. Let us recall that a nonlinearity $g$ is said to satisfy the  Ambrosetti-Rabinowitz condition if
$$
\text{There exist } \theta>2 \text{ and } s_1>0 \text{ such that }
0<\theta G(s)\leq sg(s) \quad \forall s \geq s_1,
\leqno{(\mathcal{AR})}
$$
where $G(s)=\int_0^s g(t)\ dt$.  This condition is known to be central when proving that Palais-Smale sequences are bounded.  When the domain $\Omega \subset \R^N$ is bounded and the nonlinearity is subcritical, the boundedness leads directly to the strong convergence of a subsequence. \medskip

In the case where the superquadratic term is positive, many efforts have been done to weaken the condition $(\mathcal{AR})$. However, to the best of our knowledge, this issue has not been considered for functionals of the type
$$J(u)=\int_\Omega \left(\frac{1}{2} |\nabla u|^2 - c(x)G(u)\right), \quad u \in H_0^1(\Omega) $$
when $c$ changes sign and $g$ is a superlinear function not satisfying $(\mathcal{AR})$. A typical example of such a nonlinearity is $g(s)=s\ln (s+1)$.

When $g(s)=s^{p-1}$ with $p \in [2, 2^*)$, using the homogeneity of $g$ it is straightforward that $J$ satisfies the Palais-Smale condition. When $g$ is not powerlike, this issue becomes delicate, as shown in \cite{AT} (see also \cite{ADP}), where the authors assume that $g$ is superlinear and asymptotically powerlike at infinity, i.e.
$$\text{There exist } p>2 \text{ such that } \lim_{u \to \infty} \frac{g(u)}{u^{p-1}}=1.\leqno{(\mathcal{G})}$$
Note that this condition implies $(\mathcal{AR})$. Furthermore, in \cite{AT} one needs to assume the so called {\it thick zero set} condition on $c \in \mathcal{C}(\overline{\Omega})$:
$$\overline{(\Om_+)} \cap \overline{(\Om_-)}=\emptyset, \leqno{(\mathcal{AT})}
$$
where
 $$\Omega_+:= \{x \in \Omega; \ c(x)>0\} \quad \text{and} \quad \Omega_-:= \{x \in \Omega; \ c(x)<0\}.$$
In \cite{RTT}, still under $(\mathcal{G})$, the authors were able to remove $(\mathcal{AT})$, but at the expense of some alternative strong conditions on $c$. \medskip

In our problem we prove that the Palais-Smale condition is satisfied without assuming $(\mathcal{AT})$ nor any special condition on $c$.  Given $V \in L^q(\Omega)$, with $q>\frac{N}{2}$, we denote by $\lambda_1(V)$ the first eigenvalue of the problem
$$ \left\{
\begin{array}{lll}
-\Delta u +V(x)u=\lambda u & {\rm in } & \Omega,\\
          u  = 0    & {\rm on } & \partial \Omega.
\end{array}\right. 
$$
Let us recall that $\lambda_1(V)$ is given by
$$\lambda_1(V)=\inf_{\|u\|_2=1} \int_\Omega \left( |\nabla v|^2+V(x)v^2\right).$$
It is well-known that $\lambda_1(V)$ is simple, so that it is achieved by an unique $\varphi_1>0$ such that 
$\|\varphi_1\|_2=1$.

Our main result is the following:
\begin{thm}
\label{t1}
Assume $(\mathcal{H})$ and $c^+ \not \equiv 0$. If $ \lambda_1 (-c - \mu f)>0$ then $(P)$ has two positive solutions.
\end{thm}

\begin{rem}
In \cite[Theorem 2]{JS}, assuming $c \gneqq 0$, it is proved that if 
\begin{equation}\label{BestC}
 \|\mu f\|_{\frac{N}{2}}<C_N,
\end{equation}
where  $C_N$  denotes the best Sobolev constant for the embedding $H_0^1(\Omega) \subset L^{2^*}(\Omega)$, then
there exists $\overline{c} >0$ such that $(P)$ has at least two bounded solutions if $\|c\|_q<\overline{c}$. We observe that under (\ref{BestC}), taking $||c||_q$ small enough one has 
$\lambda_1 (-c - \mu f) >0. $ Thus our condition is more general, even when $c \geq 0$.  We point out however that $f$ is allowed to be sign changing in \cite{JS}. 
\end{rem}

In \cite[Remark 3.2]{ArDeJeTa1} it is shown that when $c \equiv 0$, $(P)$ has a solution if and only if $\lambda_1(- \mu f) >0$. We now complement this result.

\begin{lem}\label{optimal}
Assume $(\mathcal{H})$.
\begin{enumerate}
\item If $c \geq 0$ then $\lambda_1(-c - \mu f) >0$ is necessary  for $(P)$ to have a non-negative solution.
\item $\lambda_1(-c)> 0 $ is necessary for $(P)$ to have a non negative solution and under this condition every solution of $(P)$ is non-negative.
\end{enumerate}
\end{lem}

\begin{rem}
As far as non negative solutions are concerned, Lemma \ref{optimal} (1) shows that when $c \geq 0$ the condition $\lambda_1(-c - \mu f) >0$ is necessary in Theorem \ref{t1}. However, other kinds of solutions of $(P)$, namely negative or sign-changing solutions, may exist if $\lambda_1(-c - \mu f) \leq 0$. See \cite{DeJe} in this direction.
\end{rem}


This paper is organized as follows.  In Section $2$ we prove some preliminary results and show that the functional $I$ has the geometry described above. Section $3$ is devoted to the Cerami condition for $I$. Finally in  Section $4$ we prove Theorem \ref{t1}, Lemma \ref{optimal}  and we give a condition under which $\lambda_1 (-c - \mu f) >0$ holds, see Lemma \ref{lf}.

\subsection{Notation}
\begin{itemize}

\item The Lebesgue norm in $L^r (\Omega)$ will be denoted by $\| \cdot
\|_r$ and the usual norm of $H_0^1(\Omega)$ by $\|\cdot \|$, i.e. 
$\|u\|=\|\nabla u\|_2$. The Holder conjugate of $r$ is denoted by $r'$.
\item The weak convergence is denoted by $\rightharpoonup$. 
\item The positive
and negative parts of a function $u$ are defined by $u^{\pm} :=\max
\{\pm u,0\}$. 
\item If $U \subset \R^N$ then we denote its interior by int $U$ and its closure by $\overline{U}$. The characteristic function of $U$ is denoted by $\chi_U$.
\item We denote by $B(0,R)$ the ball of radius $R$ centered at $0$ in $H_0^1(\Omega)$.
\end{itemize}

\medskip
\section{Preliminaries}
\medskip

\begin{lem}
\label{pq}
Assume $(\mathcal{H})$.
\strut
\begin{enumerate}
\item If $v$ is a non-negative  solution of $(Q)$ then $u = \mu^{-1}ln(1+ \lambda v)$ is a non-negative  solution of $(P)$. Similarly if $u$ is a non negative solution of $(P)$ then $v$ given by (\ref{cv}) is a non negative solution of $(Q)$.\smallskip
\item If $v$ is a critical point of $I$  then $v$ is a non-negative solution of $(Q)$. \smallskip
\item If $u$ is a non-negative solution of $(P)$ then $u$ is positive.
\end{enumerate}

\end{lem}

\begin{proof}
 Let $v\geq 0$ be a solution of $(Q)$.
From the expression of $g_\lambda$ it is seen that $v$ solves
\begin{equation}
\label{eqv}
-\Delta v=\frac{1}{\lambda}c(x) (1+\lambda v)\ln (1+\lambda v)+\frac{\mu}{\lambda} f(x)(1+\lambda v).
\end{equation}
Let $u=\frac{1}{\mu} \ln(1+\lambda v)$, i.e. $\exp(\mu u)=1+\lambda v$. Since $v \geq 0$ and $\nabla u=\frac{\lambda}{\mu}\frac{\nabla v}{1+\lambda v}$, one may easily see that $u\in H_0^1(\Omega)$. If $\phi \in H_0^1(\Omega)$ then $\psi=\frac{\phi}{1+\lambda v} \in H_0^1(\Omega)$, so that \eqref{eqv} provides
\begin{eqnarray}
\label{eqv2}
\nonumber
\int_{\Omega} \nabla v \nabla \psi &=&\frac{1}{\lambda}\int_\Omega c(x)\psi (1+\lambda v)\ln (1+\lambda v)  +\frac{\mu}{\lambda} \int_{\Omega} f(x)\psi(1+\lambda v)\\
&=&\frac{1}{\lambda}\int_\Omega c(x)\phi \ln (1+\lambda v) +\frac{\mu}{\lambda} \int_{\Omega} f(x)\phi.
\end{eqnarray} 
Now, from $\nabla v=\frac{\mu}{\lambda} \exp(\mu u) \nabla u$ and $\nabla \psi=\frac{\nabla \phi}{1+\lambda v}-\frac{\lambda \phi \nabla v}{(1+\lambda v)^2}$, we get

\begin{eqnarray*}
\int_{\Omega} \nabla v \nabla \psi &=& \frac{\mu}{\lambda} \int_\Omega \exp(\mu u) \nabla u \left(\frac{\nabla \phi}{1+\lambda v}-\frac{\lambda \phi \nabla v}{(1+\lambda v)^2}\right)
=\frac{\mu}{\lambda} \int_\Omega \nabla u \left(\nabla \phi -\frac{\lambda \phi \nabla v}{1+\lambda v}\right)\\
&=&
\frac{\mu}{\lambda} \int_\Omega \nabla u \left(\nabla \phi -\frac{\mu \phi \exp(\mu u) \nabla u}{1+\lambda v}\right)=\frac{\mu}{\lambda} \int_\Omega \left(\nabla u \nabla \phi -\mu |\nabla u|^2 \phi\right).
\end{eqnarray*}
Futhermore, we have
$$\frac{1}{\lambda}\int_\Omega c(x)\phi \ln (1+\lambda v)=\frac{\mu}{\lambda} \int_\Omega c(x)u \phi,$$
so we deduce from \eqref{eqv2} that $u$ is a solution of $(P)$. By similar arguments we prove the reverse statement. This proves (1). \\

\noindent To prove (2), let $v$ be a critical point of $I$. Then
\begin{equation}
\label{ev1}
\int_{\Omega} \left[\nabla v \nabla \varphi -(c(x)+\mu f(x))v^+ \varphi\right]  - \int_{\Omega}c(x)g_\lambda(v^+)\varphi  -\frac{\mu}{\lambda} \int_{\Omega} f(x)\varphi=0
\end{equation}
 for all $\varphi \in H_0^1(\Omega)$. Taking $\varphi=-v^-$ we get
$$\int_\Omega |\nabla v^-|^2+\frac{\mu}{\lambda} \int_\Omega f(x)v^-=0.$$
Since $f\geq 0$, we get $$\int_\Omega |\nabla v^-|^2\leq 0$$
and it follows that $v^-\equiv 0$, i.e. $v\geq 0$.  The proof that $v \in L^{\infty}(\Omega)$ can be found in \cite[Lemma 13]{JS}, so we omit it.\\

Finally, if $u \geq 0$ is a solution of $(P)$ then, since $\mu >0$ and $f \geq0$, $u$ is a bounded weak supersolution of
$$-\Delta u =c(x)u, \quad u \in H_0^1(\Omega).$$
By a  standard argument relying on the Harnack inequality, see  \cite[Theorem 1.2]{Tr}, we have either $u \equiv 0$ or $u>0$. Since $f \gneqq 0$, we get $u>0$.
\end{proof}

We shall prove now that when $\lambda_1 (-c - \mu f) >0$ the functional $I$ takes positive values on a sphere centered at the orign if $\lambda$  is sufficiently large. To this aim we first  analyze the behavior of $g_\lambda$ with respect to $s$ and $\lambda$.
Note that for any $\lambda>0$ fixed, $g_\lambda$ is superlinear at $ + \infty $, i.e.
\begin{equation}\label{superlinear}
\frac{g_\lambda(s)}{s} \rightarrow \infty \quad \text{ as } \quad s \to + \infty.
 \end{equation}
Furthermore $g_\lambda$ is uniformly superlinear at $0$, i.e. 
\begin{equation}
\label{us}
\frac{g_\lambda(s)}{s} \rightarrow 0 \quad \text{ as } \quad s \to 0 \quad \text{ uniformly with respect to } \lambda.
\end{equation}
We provide now a global estimate on $G_\lambda$ involving $s$ and $\lambda$:
\begin{lem}
\label{lg}
Let $p>1$. Given $\varepsilon>0$ there exists $C_\varepsilon>0$ such that 
$$G_\lambda(s)\leq \varepsilon s^2 +C_\varepsilon (1+\ln \lambda)s^{p+1}, \quad \forall s>0, \ \forall \lambda \geq 1.$$
\end{lem}

\begin{proof}
It is enough to show that given $\varepsilon>0$ there exists $C_\varepsilon>0$ such that 
$$g_\lambda(s)\leq \varepsilon s+C_\varepsilon (1+\ln \lambda)s^p, \quad \forall s>0, \ \forall \lambda \geq 1.$$
From \eqref{us} there exists $s_0>0$ such that $$g_\lambda(s)\leq\varepsilon s, \ \forall s \in [0,s_0], \ \forall \lambda>0.$$
Now, if $\lambda\geq 1$ then, $\forall s >0$
\begin{eqnarray}
g_\lambda(s)&\leq&\frac{1}{\lambda} \left[ \lambda(1+s)\ln (\lambda(1+s))\right]= (\ln \lambda )(1+s)+(1+s)\ln(1+s) \nonumber\\
&\leq & (1+\ln \lambda)\left[(1+s)+(1+s)\ln(1+s)\right].
\end{eqnarray}
Since $$\frac{(1+s)+(1+s)\ln(1+s)}{s^p} \rightarrow 0 \quad \text{ as } \quad s \to \infty,$$ there exists $s_1>s_0$ such that 
$$(1+s)+(1+s)\ln(1+s)\leq \varepsilon s^p, \quad \forall s\geq s_1.$$
Finally, by continuity, there exists $C_\varepsilon>0$ such that
$$(1+s)+(1+s)\ln(1+s)\leq C_\varepsilon s^p, \quad \forall s\in [s_0,s_1].$$
Taking $C_\varepsilon$ sufficiently large, we get
$$g_\lambda(s)\leq \varepsilon s+C_\varepsilon (1+\ln \lambda)s^p, \quad \forall s>0, \ \forall \lambda \geq 1.$$
\end{proof}


\begin{lem}
\label{l1}
Let $V \in L^{q}(\Omega)$, with $q>\frac{N}{2}$. If $\lambda_1(V)>0$ then there exists $K_1>0$ such that 
\begin{equation}
\label{co}
\int_\Omega \left( |\nabla v|^2+V(x)(v^+)^2\right)\geq K_1 \|v\|^2 \quad \forall v \in H_0^1(\Omega).
\end{equation}
\end{lem}

\begin{proof}
Let us first prove that there exists a constant $K_1>0$ such that 
\begin{equation}
\label{i1}
\int _\Omega\left( |\nabla v|^2+V(x)v^2\right)\geq K_1 \|v\|^2 \quad \forall v \in H_0^1(\Omega).
\end{equation}  Indeed, assume by contradiction that there is a sequence $(v_n) \subset H_0^1(\Omega)$ such that
$$\int_{\Omega}\left( |\nabla v_n|^2+V(x)(v_n)^2\right) \leq \frac{\|v_n\|^2}{n}.$$
Setting $w_n=\frac{v_n}{\|v_n\|}$ we may assume that, up to a subsequence, 
$$w_n \rightharpoonup w_0 \mbox{ in }  H_0^1(\Omega) \quad \mbox{ and } \quad w_n \rightarrow w_0 \mbox{ in } L^r(\Omega) \mbox{ for } r \in [1,2^*).$$
 In particular since $q > \frac{N}{2}$ we have that $w_n \rightarrow w_0$ in $L^{2q'}(\Omega)$. Thus from
\begin{equation}\label{x1}
\int_{\Omega}\left( |\nabla w_n|^2+V(x)(w_n)^2\right) \leq \frac{1}{n}
\end{equation}
it follows that 
\begin{equation}\label{x2}\int_{\Omega}\left( |\nabla w_0|^2+V(x)(w_0)^2\right)\leq 0.\end{equation} We claim that $w_0 \not \equiv 0$. Indeed, if $w_0 \equiv 0$ then $w_n \rightarrow 0$ in $L^{2q'}(\Omega)$ and \eqref{x1} yields $w_n \rightarrow 0$ in $H_0^1(\Omega)$, which is impossible since $\|w_n\|=1$. Hence $w_0 \not \equiv 0$ and consequently 
\eqref{x2} provides $\lambda_1(V)\leq 0$, which contradicts our assumption. Thus \eqref{i1} is proved.
Finally, we may assume that $K_1\leq 1$, so that
\begin{eqnarray*}
\int_{\Omega} \left( |\nabla v|^2+V(x)(v^+)^2\right)&=&\int_{\Omega} |\nabla v^-|^2 + \int_{\Omega}\left( |\nabla v^+|^2+V(x)(v^+)^2\right)\\ &\geq & \|v^-\|^2 + K_1\|v^+\|^2 \geq K_1\|v\|^2.
\end{eqnarray*}
\end{proof}


We are now ready to prove that $I$ has the appropriate geometry when $\lambda$ is large enough:
\begin{prop}
\label{p1}
Assume that $\lambda_1(-c-\mu f)>0$. There exists $\lambda_0>0$ such that if $\lambda\geq \lambda_0$ then, for some $M_\lambda,R_\lambda>0$, there holds $I(v)\geq  M_\lambda$ if $\|v\|=R_\lambda$. 
\end{prop}

\begin{proof}
Since $\lambda_1(-c-\mu f)>0$, by Lemma \ref{l1} there exists $K_1>0$ such that $$\int_{\Omega}\left( |\nabla v|^2-[c(x)+\mu f(x)](v^+)^2\right)\geq K_1 \|v\|^2 \quad \forall v \in H_0^1(\Omega).$$
Let $p>1$, close to $1$, satisfy $(p+1)q' < 2^*$. This choice of $p$ is possible since $q > \frac{N}{2}$. By Lemma \ref{lg}, given $\varepsilon>0$ there is $C_\varepsilon>0$ such that for every $\lambda\geq 1$ we have
$$I(v)\geq K_1\|v\|^2-C_1\varepsilon \|v\|^2 -C_\varepsilon (1+\ln \lambda)\|v\|^{p+1} -\frac{\|\mu f\|_q}{\lambda}\|v\|,$$
for some $C_1>0$.
Hence $$I(v)\geq (K_1-C_1\varepsilon)R^2-C_\varepsilon (1+\ln \lambda)R^{p+1}-\frac{\|\mu f\|_q}{\lambda}R$$
if $\|v\|=R$.

We fix $\varepsilon>0$ such that $K_2:=K_1-C_1 \varepsilon>0$ and set $R=\lambda^{-\theta}$ with $\theta \in (0,1)$.
Then
$$I(v)\geq \lambda^{-2\theta}\left(K_2-C_\varepsilon (1+\ln \lambda)\lambda^{\theta(1-p)}-\lambda^{\theta-1}\|\mu f\|_q\right).$$
Since $p>1>\theta>0$, we have $$(1+\ln \lambda)\lambda^{\theta(1-p)} \rightarrow 0 \quad \text{and} \quad \lambda^{\theta-1} \rightarrow 0 \quad \text{as} \quad \lambda \to \infty.$$
Thus $$I(v)\geq M_\lambda>0$$ if $\lambda$ is sufficiently large and $\|v\|=R_\lambda:=\lambda^{-\theta}$.
\end{proof}

\medskip

\section{The Palais-Smale condition}
\medskip

\noindent From now on, we fix $\lambda = \lambda_0$ where $\lambda_0$ is given by Proposition \ref{p1}. We set 
$$\alpha_c=\inf\left\{ \int_{\Omega} \left(|\nabla u|^2 -\mu f(x) (u^+)^2\right); u\in H_0^1(\Omega), \ \|u\|_2=1, \ cu^+ \equiv 0\right\}.$$
In the next proposition, we shall use an explicit expression of $G_\lambda$, namely,
\begin{eqnarray}
\label{eg}
\nonumber G_\lambda(s)&=&\frac{1}{4\lambda^2}\left[(1+\lambda s)^2\left(2\ln(1+\lambda s)-1\right)+1\right]-\frac{s^2}{2}
\\
\nonumber &=&\frac{1}{\lambda^2}\left[ \frac{(\lambda s)^2}{2} \ln(1+\lambda s)- \frac{(\lambda s)^2}{4} -\frac{\lambda s}{2} +\frac{1}{2} \ln(1+ \lambda s)+\lambda s\ln(1+\lambda s)\right] -\frac{s^2}{2}\\
&=&\frac{s^2}{2}\ln(1+\lambda s) -\frac{s^2}{4} -\frac{s}{2\lambda} +\frac{1}{2\lambda^2}\ln(1+\lambda s)  +\frac{s}{\lambda} \ln(1+\lambda s)  -\frac{s^2}{2}
\end{eqnarray}
for $s>0$.

\begin{prop}
\label{cerami}
If $\alpha_c>0$ then $I$ satisfies the Palais-Smale condition.
\end{prop}

\begin{proof}
Let $(u_n)$ be a Palais-Smale sequence for $I$ at the level $d \in \R$, i.e.
\begin{equation}\label{Cer}
I(u_n)\rightarrow d \quad \text{ and } \quad  \|I'(u_n)\|_*\rightarrow 0.
\end{equation} 
From (\ref{Cer}) we have
\begin{equation}\label{e2}
\frac{1}{2}\int_{\Omega} \left[|\nabla u_n|^2 -(c(x)+\mu f(x)) (u_n^+)^2\right]- \int_{\Omega}c(x)G_\lambda(u_n^+)  -\frac{\mu}{\lambda}\int_{\Omega} f(x)u_n= d+o(1)
\end{equation}
and  \begin{equation}\label{e3}\left|\int_{\Omega} \left[\nabla u_n \nabla \varphi -(c(x)+\mu f(x))u_n^+ \varphi\right]  - \int_{\Omega}c(x)g_\lambda(u_n^+)\varphi  -\frac{\mu}{\lambda} \int_{\Omega} f(x)\varphi \right|
\leq \varepsilon_n \|\varphi\|
\end{equation}
for some sequence $\varepsilon_n \rightarrow 0$ and for every $\varphi \in H_0^1(\Om)$.
In particular, we have
\begin{equation}\label{ceramieffect}
\left| \langle I'(u_n), u_n \rangle \right|\leq \varepsilon_n \|u_n\|.
\end{equation}
Let us assume that $\|u_n\|\rightarrow \infty$ and set $v_n=\frac{u_n}{\|u_n\|}$. Up to a subsequence,
we have $$v_n \rightharpoonup v_0 \text{ in } H_0^1(\Om), \quad v_n \rightarrow v_0 \text{ in } L^r(\Om),\ \forall r \in [1,2^*), \quad \text{ and }\quad 
v_n \rightarrow v_0 \text{ {\it a.e.} in }\Om.$$ We claim that $cv_0^+\equiv 0$.
Indeed, from \eqref{e3}  we have, using the convergences above,

\begin{equation}\label{e4} \int_{\Omega}c(x)\frac{g_\lambda(u_n^+)}{\|u_n\|}\varphi =\int_{\Omega} \left[\nabla v_0 \nabla \varphi -(c(x)+\mu f(x))v_0^+ \varphi\right] +o(1) < \infty, \end{equation}
for every $\varphi \in H_0^1(\Om)$.
If $cv_0^+ \not \equiv 0$ then we may choose $\varphi \in H_0^1(\Om)$ and a measurable subset $\Om_\varphi \subset \Om$ such that 
$$|\Om_\varphi|>0, \quad cv_0^+\varphi >0 \text{ on } \Om_\varphi \subset \Om, \quad \text{and} \quad 
cv_0^+\varphi =0 \text{ on } \Om \setminus \Om_\varphi.$$
Now, from (\ref{superlinear}), we have 
$$\liminf c(x)\frac{g_\lambda(u_n^+)}{\|u_n\|}\varphi=\liminf c(x)v_n^+ \frac{g_\lambda(\|u_n\|v_n^+)}{\|u_n\|v_n^+} \varphi = + \infty \quad \text{on} \quad \Om_\varphi.$$
Fatou's lemma then yields a contradiction with \eqref{e4}. Therefore $cv_0^+ \equiv 0$, i.e. $v_0^+=0$ in $\Omega_+ \cup \Omega_-$.
On the other hand, taking $\varphi=v_0$ in \eqref{e3} and dividing it by $\|u_n\|$ we get
$$\int_{\Omega} \left[\nabla v_n \nabla v_0 -(c(x)+\mu f(x))v_n^+ v_0 \right] \rightarrow 0,$$
so that, using $v_n \rightharpoonup v_0$ in $H_0^1(\Omega)$ and $cv_0^+ \equiv 0$, we get
$$\int_{\Omega} \left[ |\nabla v_0|^2 -\mu f(x)(v_0^+)^2\right]=0.$$ Thus $v_0 \equiv 0$ (otherwise $\alpha_c\leq 0$). Now from \eqref{e3} we have, taking $\varphi = u_n$ and using the definition \eqref{gl} of $g_{\lambda}$,
\begin{equation}\label{ee5}
\left| \int_{\Omega} ( |\nabla u_n|^2 - \mu f(x))(u_n^+)^2 - \frac{1}{\lambda} \int_{\Omega} c(x) (1 + \lambda u_n^+) \ln (1 + \lambda u_n^+) u_n^+ - \frac{\mu}{\lambda}\int_{\Omega} f(x) u_n^+ \right| \leq \varepsilon_n \|u_n\|.
\end{equation}
Dividing by $||u_n||^2$ and using that $v_n \to 0$ in $L^r(\Omega),  \forall r \in [1, 2^*)$ we get
$$1 - \int_{\Omega} c(x) (v_n^+)^2 \ln (1 + \lambda ||u_n||v_n^+) \to 0.$$
Now, using the property $\ln(st) = \ln s + \ln t$, it follows that
$$ 1 - \ln (||u_n||) \int_{\Omega} c(x) (v_n^+)^2 - \int_{\Omega} c(x) (v_n^+)^2 \ln \Big(\lambda v_n^+ + \frac{1}{||u_n||}\Big) \to 0.$$
We claim that
\begin{equation}\label{ee6} \ln (||u_n||) \int_{\Omega}c(x) (v_n^+)^2 \rightarrow 0.
\end{equation}
In that case we would get
$$
\int_{\Omega}c(x) (v_n^+)^2 \ln \left(\lambda v_n^+ + \frac{1}{||u_n||}\right) \rightarrow 1,$$
which clearly contradicts the fact that $v_0 =0$. To prove \eqref{ee6} we define for every $s >0$
$$H_{\lambda}(s) = \frac{1}{2}g_{\lambda}(s) s - G_{\lambda}(s).$$
From \eqref{gl} and \eqref{eg} it follows that
\begin{equation}\label{eee6} H_\lambda(s) = \frac{s^2}{4}+\frac{s}{2\lambda}-\frac{s}{2\lambda}\ln (1+\lambda s)-\frac{1}{2\lambda^2}\ln (1+\lambda s) 
\end{equation} 
From \eqref{ceramieffect} we get
$$I(u_n) - \frac{1}{2}\langle I'(u_n), u_n \rangle =c+\varepsilon_n \|u_n\|+o(1),$$
which leads, using the definition of $H_{\lambda}$, to
\begin{equation}\label{ee7}
\int_{\Omega} c(x) H_{\lambda}(u_n^+) - \frac{1}{2}\int_{\Omega}f(x) u_n =c+\varepsilon_n \|u_n\|+o(1).
\end{equation}
Now, combining \eqref{eee6} and \eqref{ee7}, we obtain 
\begin{eqnarray*}
\frac{1}{4} \int_{\Omega} c(x) (u_n^+)^2 &=& c +\varepsilon_n \|u_n\|+ \frac{1}{2 \lambda} \int_{\Omega} c(x)  u_n^+ - \frac{1}{2 \lambda} \int_{\Omega} c(x) u_n^+ \ln (1+ \lambda u_n^+)  \\
&& + \frac{1}{2 \lambda^2 } \int_{\Omega} c(x) \ln (1+ \lambda u_n^+) + \frac{1}{2}\int_{\Omega} f(x) u_n + o(1).
\end{eqnarray*}
Hence
\begin{eqnarray*}
\ln (||u_n||) \int_{\Omega} c(x) (v_n^+)^2 &=& 4  \frac{\ln ||u_n||}{||u_n||^2}\Big( c + \varepsilon_n \|u_n\|+\frac{1}{2 \lambda} \int_{\Omega}c(x) u_n^+ - \frac{1}{2 \lambda} \int_{\Omega}c(x) u_n^+ \ln(1+ \lambda u_n^+) \\
&  & + \frac{1}{2 \lambda^2} \int_{\Omega} c(x) \ln (1 + \lambda u_n^+) + \frac{1}{2} \int_{\Omega} f(x) u_n + o(1) \Big) \rightarrow 0.
\end{eqnarray*}
Thus \eqref{ee6} is proved and we reach a contradiction. Therefore $(u_n)$ must be bounded and, up to subsequence, we have $u_n \rightharpoonup u_0$ in $H_0^1(\Omega)$ and $u_n \rightarrow u_0$ in $L^p(\Omega)$ for $p \in [1,2^*)$. At this point the strong convergence follows in a standard way. We refer to \cite[Lemma 11]{JS} for a proof.
\end{proof}

\begin{cor}
If $\lambda_1(-c-\mu f)>0$ then $I$ satisfies the Palais-Smale condition.
\end{cor}

\begin{proof}
Let $\|u\|_2=1$ with $cu^+ \equiv 0$. Since $\lambda_1(-c-\mu f)>0$, by Lemma \ref{l1} there is a constant $K_1>0$ such that
\begin{eqnarray*}
 \int_{\Omega} \left(|\nabla u|^2 -\mu f(x)(u^+)^2\right)&=&\int_{\Omega} \left(|\nabla u|^2 -(c(x)+\mu f(x))(u^+)^2\right)\\&\geq &K_1\|u\|^2\geq  SK_1\|u\|_2^2=SK_1>0,
\end{eqnarray*}
where $S$ is the best Sobolev constant for the embedding $H_0^1(\Omega) \subset L^2(\Omega)$. Thus $\alpha_c>0$ and by Proposition \ref{cerami} we get the conclusion.
\end{proof}

\medskip
\section{Proof of Theorem \ref{t1} and Lemma  \ref{optimal}}
\medskip

We are now ready to prove our main results.\\

\noindent {\it Proof of Theorem \ref{t1}:}
By Proposition \ref{p1} we know that if $\lambda_1(-c-\mu f)>0$  and $\lambda = \lambda_0$  then there are $M_\lambda,R_\lambda>0$ such that $I(v)\geq  M_\lambda$ if $\|v\|=R_\lambda$. Moreover, it easily seen that if $f \not \equiv 0$ then $I$ takes negative values in the ball $B(0, R_{\lambda})$.
Therefore, by weak lower semi-continuity, we infer that  the infimum of $I$ in $B(0, R_{\lambda})$ is achieved by some $v_1\not \equiv 0$, which is a critical point  of $I$. Furthermore, since  $G_{\lambda}(s)/s^2 \to \infty$ as $s \to \infty$, if $v$ is a smooth function supported in $\Omega_+$ then $I(tv)\rightarrow -\infty$ as $t \to \infty$. We fix $t>0$ and $v$ such that $v_0=tv$ satisfies $\|v_0\|>R_\lambda$ and $I(v_0)<0$. Now let
$$\Gamma:=\{\gamma \in \mathcal{C}([0,1],H_0^1(\Om)); \ \gamma(0)=0, \gamma(1)=v_0\}$$
and $$d:=\inf_{\gamma \in \Gamma} \max_{t \in [0,1]} I(\gamma(t)).$$
Since $I$ satisfies the Palais-Smale condition, by the mountain-pass theorem it is straightforward 
that $I$ has a critical point $w_1$, which, by Proposition \ref{p1}, satisfies $I(w_1) =d>0$. In particular, we have $v_1 \neq w_1$.  Finally, from Lemma \ref{pq}, we know that these two critical points provide two positive solutions of $(P)$.
\qed \\

\begin{proof}[Proof of Lemma \ref{optimal}]

By Lemma \ref{pq}, we know that if $u \geq 0$ is a solution of $(P)$ then $v \geq 0 $ given by \eqref{cv}  is a solution of $(Q)$. Taking $\phi >0$, the first positive eigenfunction associated to $\lambda_1(-c - \mu f)$, as test function and using the property that $g_{\lambda} \geq 0$ on $\R$ we obtain
$$\int_\Omega \left( \nabla u \nabla \phi -c(x)u\phi - \mu f(x) u \phi \right)= \int_\Omega \left(c(x) g_{\lambda}(v) \phi + \frac{\mu }{\lambda}f(x) \phi \right)>0,$$
so that
$$\lambda_1(-c - \mu f)\int_\Omega u \phi >0.$$
Thus  $\lambda_1(-c - \mu f) >0$ and this proves the first point. \medskip

Similarly, let $\varphi>0$ be an eigenfunction associated to $\lambda_1(-c)$ and  assume that $u\geq 0$ is a solution of $(P)$.  Taking $\varphi >0$ as test function we get
$$\int_\Omega \left( \nabla u \nabla \varphi -c(x)u\varphi\right)= \int_\Omega \left(\mu |\nabla u|^2 \varphi +f(x)\varphi\right)>0.$$
Thus $$\lambda_1(-c) \int_\Omega u\varphi>0,$$ so that $\lambda_1(-c)>0$. Finally, let $u$ be a solution of $(P)$. Using $u^-$ as test function in $(P)$, we obtain
$$ - \int_{\Omega} (|\nabla u^-|^2 - c(x)|u^-|^2) = \int_{\Omega} (\mu |\nabla u|^2 u^- + f(x)u^-) \geq 0.$$
Hence $$\int_{\Omega} (|\nabla u^-|^2 - c(x)|u^-|^2) \leq 0$$ and since $\lambda_1(-c) > 0$ we get $u^- \equiv 0$, i.e. $u \geq 0$.
\end{proof}

We end this paper by deriving a condition under which  $\lambda_1(-c - \mu f) >0$ holds. \medskip

Considering $c$ and $f$ as fixed, we may rewrite the condition $\lambda_1(-c-\mu f)>0$ as an explicit condition on $\mu$. To this end, we assume that $\lambda_1(-c)>0$ and set
\begin{equation}
\label{g1}\gamma_1(-c,f):=\displaystyle \inf\left\{ \int_{\Omega} \left(|\nabla u|^2-c(x)u^2\right);\ u \in H_0^1(\Omega),\ \int_\Omega fu^2=1\right\}.
\end{equation}
It is well-known, see \cite{MM}, that
$\gamma_1(-c,f)$ is the first eigenvalue of the problem
$$ \left\{
\begin{array}{lll}
-\Delta u -c(x)u=\gamma f(x)u& {\rm in } & \Omega,\\
          u  = 0    & {\rm on } & \partial \Omega.
\end{array}\right. 
$$

\begin{lem}
\label{lf}
Let $\mu>0$ and $f\gneqq 0$. Then $\lambda_1(-c-\mu f)>0$ if and only if $\lambda_1(-c)>0$ and $0<\mu<\gamma_1(-c,f)$.
\end{lem}

\begin{proof}
Note first that since $\mu f\geq 0$ we have $$\lambda_1(-c)\geq \lambda_1(-c-\mu f).$$ It is also clear that if $\lambda_1(-c)>0$ then $\gamma_1(-c,f)>0$.
Hence, if $\lambda_1(-c-\mu f)>0$ we have $\lambda_1(-c)>0$ and $\gamma_1(-c,f)>0$. Moreover
$$\int_{\Omega}\left(|\nabla u|^2 - [c(x) + \mu f(x)] u^2\right)>0$$ for every  $u \not \equiv 0$. Thus, if $\int_\Omega fu^2=1$ then 
$$\mu <\int_{\Omega}\left(|\nabla u|^2 - c(x) u^2\right).$$
Since $\gamma_1(-c,f)$ is achieved we have $\mu <\gamma_1(-c,f)$.

Let us now assume that $\lambda_1(-c)>0$ and $0<\mu<\gamma_1(c,f)$. Since $f\gneqq 0$ and $\lambda_1(-c-\mu f)$ is achieved by some $\phi>0$, we have $\int_\Omega f\phi^2>0$ and consequently
$$\frac{\int_{\Omega}\left(|\nabla \phi|^2 - c(x)\phi^2\right)}{\int_\Omega f\phi^2}\geq \gamma_1(-c,f)>\mu$$
so that $$\lambda_1(-c-\mu f)=\int_{\Omega}\left(|\nabla \phi|^2 - [c(x) + \mu f(x)] \phi^2\right)>0.$$
\end{proof}

Theorem \ref{t1} can then be restated as follows:
\begin{cor}
\label{cor1}
Assume $(\mathcal{H})$, $c^+ \not \equiv 0$ and $\lambda_1(-c)>0$. Then $(P)$ has two positive solutions for $ 0< \mu<\gamma_1(-c,f)$.
\end{cor}

\end{document}